\documentclass[11pt]{amsart}

\usepackage{epigamath}

\usepackage[english]{babel}

\numberwithin{equation}{section}

\usepackage{inputenc,enumitem}
\usepackage{hyperref}

\usepackage[all]{xy}

\usepackage{graphicx}
\usepackage{amssymb, amsfonts, mathtools}
\usepackage{mathrsfs}
\usepackage{amsthm}
\usepackage{tikz-cd}
\allowdisplaybreaks[1]

\newtheorem{thm}{Theorem}[section]

\newtheorem{lemma}[thm]{Lemma}
\newtheorem{prop}[thm]{Proposition}
\newtheorem{cor}[thm]{Corollary}

\theoremstyle{definition}
\newtheorem{defn}[thm]{Definition}

\theoremstyle{remark}

\newtheorem{remark}[thm]{Remark}

\newcommand{\Z}{\mathbb{Z}}

\DeclareMathOperator{\Hom}{Hom}
\DeclareMathOperator{\GL}{GL}
\DeclareMathOperator{\PSL}{PSL}
\DeclareMathOperator{\GSp}{GSp}
\DeclareMathOperator{\Aut}{Aut}
\DeclareMathOperator{\Gal}{Gal}
\DeclareMathOperator{\Out}{Out}
\DeclareMathOperator{\Jac}{Jac}
\DeclareMathOperator{\MD}{MD}
\newcommand{\univ}{{\operatorname{univ}}}
\DeclareMathOperator{\coker}{coker}

\newcommand{\colonequals}{:=}

\def\isolow{\vbox to 0pt{\vss\hbox{$\scriptstyle\sim$}\vskip-2pt}}
\newcommand{\isor}{\xrightarrow{\;\isolow\;}}

\newcommand{\supth}[1]{\ensuremath{#1^{\mathrm{th}}}}
\newcommand{\ab}{{\operatorname{ab}}}
\newcommand{\lra}{\longrightarrow}

\EpigaVolumeYear{7}{2023} \EpigaArticleNr{8} \ReceivedOn{October 20, 2020}
\InFinalFormOn{February 15, 2022}
\AcceptedOn{November 21, 2022}

\title{Group-theoretic Johnson classes and non-hyperelliptic curves with torsion Ceresa class}
\titlemark{Group-theoretic Johnson classes and non-hyperelliptic curves}

\author{Dean Bisogno}
\address{Department of Mathematics, Colorado State University, Fort Collins, CO 80523, USA}
\email{dbisogno@rams.colostate.edu}

\author{Wanlin Li}
\address{Massachusetts Institute of Technology, 77 Massachusetts Avenue, 
Cambridge, MA 02139, USA}
\email{wanlinli@mit.edu}

\author{Daniel Litt}
\address{Department of Mathematics, University of Toronto,  
Toronto, Ontario, 
Canada
M5S 2E4}
\email{daniel.litt@utoronto.ca}

\author{Padmavathi Srinivasan}
\address{Department of Mathematics, University of Georgia, 
Athens, GA 30602, USA}
\email{Padmavathi.Srinivasan@uga.edu}

\authormark{D.~Bisogno, W.~Li, D.~Litt, and P.~Srinivasan}

\AbstractInEnglish{Let $\ell$ be a prime and $G$ a pro-$\ell$ group with torsion-free abelianization. We produce group-theoretic analogues of the Johnson/Morita cocycle for $G$---in the case of surface groups, these cocycles appear to refine existing constructions when $\ell=2$. We apply this construction to the pro-$\ell$ \'etale fundamental groups of smooth curves to obtain Galois-cohomological analogues of these cocycles, which represent what we call the ``modified diagonal'' and ``Johnson'' classes, and discuss their relationship to work of Hain and Matsumoto 
in the case where the curve is proper. We analyze many of the fundamental properties of these classes and use them to give two examples of non-hyperelliptic curves whose Ceresa class has torsion image under the $\ell$-adic Abel--Jacobi map.}


\acknowledgement{This project started during the 2019 AMS Mathematical Research Communities: Explicit Methods in Arithmetic Geometry in Characteristic $p$. The authors thank the AMS and the organizers of the meeting.  Li is funded by the Simons Collaboration on Arithmetic Geometry, Number Theory and Computation during the time conducting this work. Litt is supported by NSF Grant DMS-2001196.} 
  
\begin{document}


\removeabove{12pt}

\maketitle

\begin{prelims}

\DisplayAbstractInEnglish

\bigskip

\DisplayKeyWords

\medskip

\DisplayMSCclass

\end{prelims}


\newpage

\setcounter{tocdepth}{1}

\tableofcontents


\section{Introduction}

Let $X$ be a smooth, projective, geometrically integral curve over a field $K$ of genus $g\geq 3$, and let $x\in X(K)$ be a rational point. One can embed $X$ in its Jacobian $\Jac(X)$ via the Abel--Jacobi map $P \mapsto [P- x]$;  let $X^-$ denote the image of $X$ under the negation map on the group $\Jac(X)$. The Ceresa cycle is the homologically trivial algebraic cycle $X-X^-$ in $\Jac(X)$. A classical result of Giuseppe Ceresa 
\cite[Theorem 3.1]{Ceresa} shows that when $X$ is a very general curve over $\mathbb{C}$ of genus $g\ge 3$, the Ceresa cycle is not algebraically trivial.

Via the $\ell$-adic cycle class map, the Ceresa cycle gives rise to a Galois cohomology class
$$\mu(X,x) \in H^1\left(\Gal\left(\bar{K}/K\right), H^{2g-3}_{\text{\'et}}\left(\Jac(X)\otimes \bar{K},\mathbb{Z}_\ell\left(g-1\right)\right)\right)$$
which only depends on the rational equivalence class of the Ceresa cycle. Richard Hain and Makoto Matsumoto \cite{HM} reinterpret this class in terms of the Galois action on the pro-$\ell$ \'etale fundamental group of $X$ and describe an analogous class $\nu(X)$ which is basepoint-independent.

We define two classes $\MD(X, x)$ and $J(X)$ in Galois cohomology (the latter of which is basepoint-independent), called the \emph{modified diagonal} and \emph{Johnson} classes, which capture aspects of the action of Galois on the pro-$\ell$ \'etale fundamental group of $X$. Under the assumption that $X$ is smooth and projective, these classes are closely related to $\mu(X,x)$ and $\nu(X)$. The main novelty of our construction is that it proceeds via abstract group theory. In particular, it works for any pro-$\ell$ group with torsion-free abelianization---for example, we do not require our curves to be proper, and many of our results hold for general Demuskin groups. Even in the case of pro-$\ell$ surface groups, our analysis appears to refine existing results when $\ell=2$; for example, the classes $\MD(X, x)$ and $J(X)$ appear to give slightly more information than the classes $\mu(X, x)$ and $\nu(X)$ if $\ell=2$ (if $\ell\neq 2$, one may recover our classes from those in \cite{HM} and vice versa).

The Ceresa class is well known to be trivial if $X$ is hyperelliptic and $x$ is a rational Weierstrass point; likewise, the class $\nu(X)$ of \cite{HM} is trivial for any hyperelliptic curve. In Section \ref{sec:Fricke-Macbeath}, we use properties of the Johnson class to give what is, to our knowledge, the first known example of a non-hyperelliptic curve where $J(X)$ (and hence $\nu(X)$) is torsion. This curve is of genus $7$.

Moreover, in Section \ref{sec:dominated-torsion}, we show with Theorem \ref{thm:dominated-torsion} that any curve dominated by a curve with torsion Johnson class has torsion Johnson class as well. This can be viewed as a generalization of the fact that any curve dominated by a hyperelliptic curve is itself hyperelliptic. Using this property, we construct a non-hyperelliptic genus $3$ curve with torsion Johnson class.

\begin{thm}[Proposition~\ref{FrickeMacbeath}, the Fricke--Macbeath curve, and Corollary~\ref{genus-3-example}]
Let $C$ be a genus $7$ curve over a field $K$ of characteristic zero such that $C_{\overline{K}}$ has automorphism group  isomorphic to $\PSL_2(8)$. The Johnson class of\, $C$ $($that is, $J(C))$ and hence the basepoint-independent Ceresa class $\nu(C)$ defined in \cite{HM} are torsion. 

If $\iota\in \Aut(C)$ is any element of order $2$, then the quotient $C/\iota$ is non-hyperelliptic of genus $3$ with $J(C/\iota)$ and $\nu(C/\iota)$ torsion.
\end{thm}

It is well known that such curves $C$ exist (see \textit{e.g.}~\cite{MR0177342}).

\begin{remark}
  It was apparently a folk expectation (see \textit{e.g.}~\cite{MO-question}) that torsion Ceresa class implies hyperellipticity; our examples show this expectation fails. Herbert Clemens has asked \cite[Question 8.5]{hain87} if, for $C$ a curve of genus $3$, having trivial Ceresa cycle modulo algebraic equivalence is equivalent to hyperellipticity. Our example (Corollary \ref{genus-3-example}) provides evidence that this question has a negative answer, at least if one interprets triviality in the \emph{rational} group of algebraic cycles modulo algebraic equivalence; the curve discussed in Corollary \ref{genus-3-example} is a natural candidate for a counterexample. Note that triviality modulo algebraic equivalence neither implies nor is implied by triviality of $J(C)$ or $\nu(C)$.
\end{remark}

\begin{remark}
After we posted this paper to arXiV, Benedict Gross explained to us that this result verifies a prediction following from the Beilinson conjectures. Namely, one may utilize the interpretation of the Fricke--Macbeath curve $C$ as a Shimura curve to compute the $L$-function of $C^3$; the $L$-value controlling the rank of the Chow group in which the Ceresa cycle lives is non-zero, and hence this group is predicted to have rank zero. Hence the Ceresa cycle  is necessarily torsion in the Chow group. Congling Qiu and Wei Zhang \cite{QZ} later give an unconditional proof of Gross's claim. Their argument uses automorphic methods and takes as input the fact that there are no $\Aut(C)$-invariant trilinear forms on the holomorphic differentials of $C$ combined with the arithmetic input that the Jacobian of $C$ has Mordell--Weil rank zero over the totally real subfield of $\mathbb{Q}(\zeta_7)$. See also \cite{LS} for related discussion and results. 
\end{remark}

\begin{remark}
Chao Li pointed us to Jaap Top's thesis \cite[Section 3.1]{Top}, which mentions an example of a non-hyperelliptic curve of genus $3$ with torsion Ceresa cycle modulo algebraic equivalence, constructed by Chad Schoen---the example seems never to have been published. In personal communication, Schoen has indicated to us that he does not recall the example. After seeing this paper on arXiV, Arnaud Beauville \cite{Bea1} discovered a non-hyperelliptic genus $3$ curve (different from the curve in Corollary~\ref{genus-3-example}) with torsion Ceresa class in the intermediate Jacobian. Bert van Geemen alerted us that Beauville's example was Schoen's lost example---in \cite{BeaSch}, Beauville and Schoen together proved that  the Ceresa cycle of their curve is torsion modulo algebraic equivalence.
\end{remark}

\subsection*{Outline of the paper}
In Section~\ref{Sgrpthy}, we give a group-theoretic construction of the so-called modified diagonal and Johnson classes associated to a finitely generated pro-$\ell$ group with torsion-free abelianization. In Section~\ref{SCerFunGrp}, we specialize this construction to the pro-$\ell$ fundamental group of a curve and compare it to the classes $\mu(X,x)$, $\nu(X)$ of Hain--Matsumoto \cite{HM}. 
In Section~\ref{Storsion}, we study properties of this construction and apply them to give a proof of the fact that hyperelliptic curves have $2$-torsion Johnson class, and we show that any model of the Fricke--Macbeath curve has torsion Johnson/Ceresa class. We also show that any curve dominated by a curve with torsion Johnson class has torsion Johnson class itself; hence a genus $3$ non-hyperelliptic curve which is a quotient of the Fricke--Macbeath curve has torsion Johnson class as well.

\subsection*{Acknowledgements}  The authors would like to thank Jordan Ellenberg, Benedict Gross, Richard Hain, Chao Li, Bjorn Poonen, and David Stapleton for helpful conversations and suggestions. We thank the
referee for the valuable feedback and comments.

This work also owes a substantial intellectual debt to \cite{HM}, which will be apparent throughout.

\section{Group-theoretic Ceresa classes}\label{Sgrpthy}
Let $\ell$ be a prime and $G$  a non-trivial finitely generated pro-$\ell$ group with torsion-free abelianization $G^{\ab}$. 
Define the $\ell$-adic group ring of $G$ as
$$\mathbb{Z}_\ell[[G]]:=\varprojlim_{G\twoheadrightarrow H} \mathbb{Z}_\ell[H].$$
Here the inverse limit is taken over all finite groups $H$ which are continuous quotients of $G$. Let $\mathscr{I}\subset \mathbb{Z}_\ell[[G]]$ be the augmentation ideal.

\begin{prop}\label{PGabaug}
  The map $\phi\colon G\to \mathscr{I}/\mathscr{I}^2$ 
  given by
  $$\phi\colon g\longmapsto {g-1}$$
  is a continuous group homomorphism and induces an isomorphism
  $$G^{\ab}\isor \mathscr{I}/\mathscr{I}^2.$$
\end{prop}

\begin{proof}
This is \cite[Lemma 6.8.6(b)]{ProfiniteGroups}.
\end{proof}

Let $Z(G)$ denote the center of $G$. The action of $G$ on itself by conjugation gives a short exact sequence
$$1\lra G/Z(G)\lra \Aut(G)\lra \Out(G)\lra 1$$
of continuous maps of profinite groups.

\begin{defn}\label{DuniC} The \emph{modified diagonal class}, denoted by
  $$\MD_{\univ}\in H^1\left(\Aut(G), \Hom\left(\mathscr{I}/\mathscr{I}^2, \mathscr{I}^2/\mathscr{I}^3\right)\right),$$
  is the class associated to the extension of continuous $\Aut(G)$-modules 
\begin{equation}\label{EMDU}
0\lra \mathscr{I}^2/\mathscr{I}^3\lra \mathscr{I}/\mathscr{I}^3\lra \mathscr{I}/\mathscr{I}^2\lra 0.\end{equation}
The existence of $\MD_\univ$ follows from the fact that $\mathscr{I}/\mathscr{I}^2$ is a $\mathbb{Z}_\ell$-module (as $G^{\ab}$ is torsion-free by assumption). An explicit cocycle representing $\MD_\univ$ will be given in Section \ref{Section:Johnson}.
\end{defn}

\begin{remark}
We call this class the modified diagonal class because we expect that when $G$ is the pro-$\ell$ \'etale fundamental group of a curve, the Galois-cohomological avatar of $\MD_{\univ}$ (see Section \ref{SCerFunGrp}) may be written rationally as a multiple of the image of the Gross--Kudla--Schoen \cite{GrossSchoen, GrossKudla} modified diagonal cycle under an \'etale Abel--Jacobi map. See \textit{e.g.}~\cite{DarmonRotgerSols} for a Hodge-theoretic analogue of this fact. 
\end{remark}

We now proceed to find an avatar of $\MD_{\univ}$ in the cohomology of the outer automorphism group of $G$, $\Out(G)$. Geometrically, this will correspond to removing the basepoint-dependence of the class $\MD_{\univ}$ in the case where $G$ is the pro-$\ell$ \'etale fundamental group of a curve.

\subsection{Descending to \texorpdfstring{$\boldsymbol{\Out(G)}$}{Out(G)}, and the Johnson class}\label{Section:Johnson}
We first analyze the pullback of $\MD_{\univ}$ along the canonical map $G\to \Aut(G)$. We will use this analysis to construct a quotient $A(G)$ of $\Hom(\mathscr{I}/\mathscr{I}^2, \mathscr{I}^2/\mathscr{I}^3)$ such that $\MD_{\univ}|_G$ vanishes in $H^1(G, A(G))$;  hence $\MD_{\univ}$ will induce a class in $H^1(\Out(G), A(G))$, which we will term the Johnson class. The constructions here are closely related to work of Andreadakis, Bachmuth, and others (see \textit{e.g.}~\cite{andrea1, bach2, bach1}), but we include the details here as those papers deal with the discrete, rather than profinite, situation.

Note that $\mathscr{I}/\mathscr{I}^2$ is a free $\mathbb{Z}_\ell$-module by Proposition~\ref{PGabaug} and our assumption that $G^{\ab}$ is torsion-free. Tensoring the short exact sequence \eqref{EMDU} by $(\mathscr{I}/\mathscr{I}^2)^\vee$ yields 
$$0\lra \Hom \left(\mathscr{I}/\mathscr{I}^2,\mathscr{I}^2/\mathscr{I}^3\right)\lra \Hom\left(\mathscr{I}/\mathscr{I}^2,\mathscr{I}/\mathscr{I}^3\right)\lra \Hom\left(\mathscr{I}/\mathscr{I}^2,\mathscr{I}/\mathscr{I}^2\right)\lra 0.$$
The last term admits a natural map $\mathbb{Z}_\ell \hookrightarrow \Hom(\mathscr{I}/\mathscr{I}^2,\mathscr{I}/\mathscr{I}^2)$ (sending $1$ to the identity map), and pulling back along this inclusion gives a $G$-module extension
\begin{equation}\label{EdefiningX}
  0\lra \Hom \left(\mathscr{I}/\mathscr{I}^2,\mathscr{I}^2/\mathscr{I}^3\right)\lra X\lra \mathbb{Z}_\ell\lra 0,
\end{equation}
where $G$ acts trivially on $\Hom (\mathscr{I}/\mathscr{I}^2,\mathscr{I}^2/\mathscr{I}^3)$ and $\mathbb{Z}_\ell$ but non-trivially on $X$. The extension is characterized by a group homomorphism  
\begin{align*}
  G &\lra \Hom\left(\mathbb{Z}_\ell,\Hom \left(\mathscr{I}/\mathscr{I}^2,\mathscr{I}^2/\mathscr{I}^3\right)\right) \simeq \Hom \left(\mathscr{I}/\mathscr{I}^2,\mathscr{I}^2/\mathscr{I}^3\right)\\
    g & \longmapsto \left(v \mapsto g\left(\tilde{v}\right)-\tilde{v}\right),
\end{align*}
where $\tilde{v}$ is any lift of $v \in \mathbb{Z}_\ell$ to $X$.

This map factors through $G^{\ab} \cong \mathscr{I}/\mathscr{I}^2$ as $\Hom (\mathscr{I}/\mathscr{I}^2,\mathscr{I}^2/\mathscr{I}^3)$ is Abelian.

\begin{defn}\label{Dm}
For the rest of the paper, let
\begin{align*}
m \colon G^{\ab} \lra \Hom \left(\mathscr{I}/\mathscr{I}^2,\mathscr{I}^2/\mathscr{I}^3\right)
\end{align*}
be the map coming from the extension class of \eqref{EdefiningX} described in the paragraphs above. 
\end{defn}

We now give a more explicit description of the map $m$.

\begin{lemma}\label{Lmbycomult}
Consider the commutator map
\begin{align*}
  \left(\mathscr{I}/\mathscr{I}^2\right)^{\otimes 2} &\lra \mathscr{I}^2/\mathscr{I}^3\\ x\otimes y &\longmapsto xy-yx.
\end{align*}
Then the map $m$ in Definition~\ref{Dm} is the same as the map induced by adjunction:
$$
m\colon \mathscr{I}/\mathscr{I}^2\lra \Hom\left(\mathscr{I}/\mathscr{I}^2, \mathscr{I}^2/\mathscr{I}^3\right), \quad x\longmapsto \left(y\mapsto xy-yx\right)
$$
under the identification between $G^{\ab}$ and $\mathscr{I}/\mathscr{I}^2$ from Proposition \ref{PGabaug}.
\end{lemma}

\begin{proof}
  Let $X$ be as in~\eqref{EdefiningX}. Let $s\in X \subset \Hom(\mathscr{I}/\mathscr{I}^2,\mathscr{I}/\mathscr{I}^3)$ be an element reducing to the identity modulo $\mathscr{I}^2$. Then we define maps
  $$
  m_1,m_2 \colon G \lra \mathscr{I}/\mathscr{I}^2\lra \Hom\left(\mathscr{I}/\mathscr{I}^2, \mathscr{I}^2/\mathscr{I}^3\right)
  $$
  by
\begin{align*}
 m_1(g) &= \left(y \longmapsto g s(y) g^{-1} - s(y)\right),\\
 m_2(g) &= \left(y \longmapsto (g-1)s(y)-s(y)(g-1) = gs(y)-s(y)g\right).
\end{align*}
  The map $m_1$ is by definition the same as the map in Definition \ref{Dm}. The map $m_2$ is an explicit formula for the map in the statement of the lemma. Neither map depends on the choice of $s$. We wish to show they are the same.

For any $g\in G$, we have 
$$g^{-1}=\frac{1}{1+(g-1)}=1-(g-1)+(g-1)^2\bmod \mathscr{I}^3.$$ 
Hence for $g\in G$, $y\in\mathscr{I}/\mathscr{I}^2$, and $s(y) \in\mathscr{I}/\mathscr{I}^3$ being a lift of $y$, we have modulo $\mathscr{I}^3$:
\begin{align*}
    ((m_1-m_2)(g))(y) &\equiv
    gs(y)g^{-1}-s(y)-gs(y)+s(y)g\\
    &\equiv gs(y)(g^{-1}-1)-s(y)(1-g)\\
    &\equiv gs(y)((1-g)+(g-1)^2)-s(y)(1-g)\\
    &\equiv (g-1)s(y)(1-g)+gs(y)(g-1)^2\\
   &\equiv 0
\end{align*}
as $g-1\in \mathscr{I}$ and $s(y)\in \mathscr{I}/\mathscr{I}^3$ above. This shows that $m_1=m_2$, as desired.
\end{proof}

\begin{defn}\label{DAG}
Let $A(G):=\coker(m\colon \mathscr{I}/\mathscr{I}^2\to \Hom(\mathscr{I}/\mathscr{I}^2, \mathscr{I}^2/\mathscr{I}^3))$ be the cokernel of the commutator map defined above. 
\end{defn}

Using the quotient map $\Hom(\mathscr{I}/\mathscr{I}^2, \mathscr{I}^2/\mathscr{I}^3)\to A(G)$ and the inclusion $G/Z(G)\to \Aut(G)$, we get a map
$$H^1\left(\Aut(G), \Hom\left(\mathscr{I}/\mathscr{I}^2, \mathscr{I}^2/\mathscr{I}^3\right)\right)\lra H^1\left(\Aut(G), A(G)\right)\lra H^1\left(G/Z(G), A(G)\right).$$

\begin{prop}\label{J-trivial-on-G}\label{basepoint-free-prop}
The image of\, $\MD_{\univ}$ under the composition above is zero.
\end{prop}

\begin{proof}
As $G$ acts trivially by conjugation on $\mathscr{I}/\mathscr{I}^2\!$ and $\mathscr{I}^2/\mathscr{I}^3\!$, it also acts trivially on $\Hom(\mathscr{I}/\mathscr{I}^2\!\!,\mathscr{I}^2/\mathscr{I}^3)$.
This means
$H^1\!(G,\Hom(\mathscr{I}/\mathscr{I}^2\!,\mathscr{I}^2/\mathscr{I}^3)) = \Hom(G,\Hom(\mathscr{I}/\mathscr{I}^2\!,\mathscr{I}^2/\mathscr{I}^3))$.
By Lemma~\ref{Lmbycomult}, the pullback of the class $\MD_{\univ}$ in $H^1(G,\Hom(\mathscr{I}/\mathscr{I}^2,\mathscr{I}^2/\mathscr{I}^3))$ maps to  the homomorphism $m$ under this identification. But by the definition of $A(G)$, its restriction to $G/Z(G)$, and hence to $G$, is trivial.
\end{proof}

We now define the universal Johnson class.

\begin{prop}\label{prop:Juniv}
  There exists a unique element $J_{\univ}$ in $H^1(\Out(G), A(G))$ whose image in $H^1(\Aut(G), A(G))$
  under the inflation map
  $$H^1(\Out(G), A(G))\lra H^1(\Aut(G), A(G))$$
  is the same as the image of\, $\MD_{\univ}$ under the map
  $$
  H^1\left(\Aut(G), \Hom\left(\mathscr{I}/\mathscr{I}^2,\mathscr{I}^2/\mathscr{I}^3\right)\right)\lra H^1(\Aut(G), A(G))
  $$
  induced by the quotient map $\Hom(\mathscr{I}/\mathscr{I}^2,\mathscr{I}^2/\mathscr{I}^3)\to A(G)$.
\end{prop}

\begin{proof}
The definition of $A(G)$ implies that the $G/Z(G)$-action on $A(G)$ is trivial. This means we have the following  inflation-restriction exact sequence in continuous group cohomology: 
$$0\lra H^1(\Out(G), A(G))\lra H^1(\Aut(G), A(G))\lra H^1(G/Z(G), A(G))^{\Out(G)}.$$  
By Proposition \ref{basepoint-free-prop}, the image of $\MD_{\univ}$ in $H^1(G/Z(G), A(G))^{\Out(G)}$ is zero, and thus there exists a unique element $J_{\univ}$ in $H^1(\Out(G), A(G))$ whose image in $H^1(\Aut(G), A(G))$ is the same as the image of $\MD_{\univ}$.
\end{proof}

\begin{defn}\label{DJuniv}
We call the element $J_{\univ}\in H^1(\Out(G), A(G))$ constructed in Proposition \ref{prop:Juniv} the \emph{universal Johnson class}.
\end{defn}

\begin{remark}
We call this class the Johnson class because in the case where $G$ is a discrete surface group, our construction is closely related to the Johnson homomorphism studied in \cite{Johnson} and the cocycle constructed by Morita in \cite{Morita}.
\end{remark}

\subsection{The coefficient groups for the modified diagonal and Johnson classes}\label{Scoeff}
The goal of this subsection is to identify a natural $\Aut(G)$-submodule $W$ of the group $\mathscr{I}^2/\mathscr{I}^3$ such that $\MD_{\univ}$  lives in the image of the natural map
$$H^1\left(\Aut(G), \Hom\left(\mathscr{I}/\mathscr{I}^2,W\right)\right) \lra H^1\left(\Aut(G),\Hom\left(\mathscr{I}/\mathscr{I}^2,\mathscr{I}^2/\mathscr{I}^3\right)\right)$$
for $\ell\neq 2$. Similarly, we will find a natural submodule $A_W(G)\subset A(G)$ so that $J_{\univ}$ is in the image of the natural map
$$
H^1(\Out(G), A_W(G))\lra H^1(\Out(G), A(G)).
$$ 

For $\ell=2$, we will prove similar results for $2^i\MD_{\univ}$ and $2^iJ_{\univ}$, where $i=1,2$ depending on the group-theoretic properties of $G$.

\subsubsection{Preliminaries on free pro-$\boldsymbol{\ell}$ groups}

\begin{lemma}\label{Lfreeprol}
Let $G$ be a free pro-$\ell$ group, freely generated by $g_1,g_2,\ldots,g_r$, and let $\mathscr{I}$ be the augmentation ideal of the completed group ring $\Z_\ell[[G]]$.
\begin{enumerate}
\item\label{Lfreeprol-1} For each of the generators $g_i$, let $x_i \colonequals g_i-1 \in \Z_\ell[[G]]$. Then
$\mathscr{I}/\mathscr{I}^2$ is a free $\Z_\ell$-module of rank~$r$ generated by the images of\, $x_1,x_2,\ldots,x_r$, and
 $\mathscr{I}^2/\mathscr{I}^3$ is free of rank $r^2$ with basis the images of $x_ix_j$.
    \item\label{Lfreeprol-2}
    Let $H$ be another finitely generated free pro-$\ell$ group, and let $f^{\ab} \colon G^{\ab} \rightarrow H^{\ab}$ be an isomorphism. Let  $h_1,h_2,\ldots,h_r$ be any set of lifts of $f^{\ab}(g_1), \ldots, f^{\ab}(g_r)$ from $H^{\ab}$ to $H$. Then $f(g_i) = h_i$ defines an isomorphism $f \colon G \rightarrow H$.
    \item\label{Lfreeprol-3} Let $\tilde{G}$ be a finitely generated pro-$\ell$ group with torsion-free abelianization. Let  $\pi \colon G \rightarrow \tilde{G}$ be a surjection such that the induced map  $\pi^{\ab} \colon G^{\ab} \rightarrow \tilde{G}^{\ab}$ is an isomorphism. Then any automorphism $\sigma_{\tilde{G}} \colon \tilde{G} \rightarrow \tilde{G}$ lifts to an automorphism $\sigma_G \colon G \rightarrow G$.
\end{enumerate}
\end{lemma}

\begin{proof}
  \eqref{Lfreeprol-1}~
 Since $G$ is free and $\Z_\ell[[G]]$ is complete with respect to the augmentation ideal, there is by \cite[Proposition 7, p.~I-7]{Serre62} an isomorphism
\begin{equation}\label{Enciso} \Z_\ell[[G]] \isor \Z_\ell \langle \langle x_1,x_2,\ldots,x_r \rangle \rangle_{\mathrm{nc}}, \end{equation}
where $\Z_\ell \langle \langle x_1,x_2,\ldots,x_r \rangle \rangle_{\mathrm{nc}}$ is the non-commutative power series ring in $r$ variables, such that $g_i$ is sent to $x_i+1$. The claim follows.

\eqref{Lfreeprol-2}~ Since $h_1,h_2,\ldots,h_r$ are elements of $H$ whose images topologically generate $H^\ab$, by \cite[Proposition~3.9.1]{Neukirch}, it follows that $h_1,h_2,\ldots,h_r$ also generate $H$. This shows that $f$ is a surjection. We will now show that these elements in fact freely topologically generate $H$, which proves that $f$ is an isomorphism.

Note that $f^{\ab}$ also induces an isomorphism 
\[ f^{\ab} \colon G^{\ab}/(G^{\ab})^\ell \lra H^{\ab}/(H^{\ab})^\ell. \]
Combining this with \cite[Proposition~3.9.1]{Neukirch} applied to $G$ and $H$, we get that the cardinalities of the minimal generating sets for these two groups are equal since they are equal to  $\dim_{\mathbb{F}_\ell} G^{\ab}/(G^{\ab})^\ell = \dim_{\mathbb{F}_\ell} H^{\ab}/(H^{\ab})^\ell$.
Since $g_1,g_2,\ldots,g_r$ is a minimal generating set for $G$, it follows that $h_1,h_2,\ldots,h_r$ is a minimal generating set for $H$.  By \cite[Proposition 3.9.4]{Neukirch}, there are thus no relations between the $h_i$; hence $f$ is injective, as desired.

 \eqref{Lfreeprol-3}~ Choose any homomorphism $f\colon G\to G$ lifting $\sigma_{\tilde{G}}$. That it is an isomorphism follows from the previous part applied with $G=H$ and $f^{\ab} = (\pi^{\ab})^{-1} \circ \sigma_{\tilde{G}}^{\ab} \circ \pi^{\ab}$. 
\end{proof}

\begin{defn}[Alternating tensors]\label{Dnatsub} Let $G$ be a finitely generated pro-$\ell$ group with torsion-free abelianization, and let $V \colonequals \mathscr{I}/\mathscr{I}^2$. Let
  $$\iota \colon V \otimes V \lra V \otimes V$$
  be the natural involution of the $\Aut(G)$-module $V \otimes V$ that acts on a simple tensor $v_1 \otimes v_2$ as $\iota (v_1 \otimes v_2) \colonequals v_2 \otimes v_1$. Let $\textup{Alt}^2\, V \subset V \otimes V$ be the $\Aut(G)$-submodule of alternating tensors, \textit{i.e.}, the maximal submodule where $\iota$ acts as multiplication by $-1$.

Let $W \subset \mathscr{I}^2/\mathscr{I}^3$ be the image of $\textup{Alt}^2\,V$ under the natural surjective multiplication map $V \otimes V \rightarrow \mathscr{I}^2/\mathscr{I}^3$. Note that by Lemma~\ref{Lmbycomult},  since $x \otimes y-y \otimes x$ is skew-symmetric, the image of the map 
\begin{align*}m \colon \mathscr{I}/\mathscr{I}^2 &\lra \Hom\left(\mathscr{I}/\mathscr{I}^2,\mathscr{I}^2/\mathscr{I}^3\right) \\ x &\longmapsto (y \longmapsto xy-yx)
\end{align*}
in Definition~\ref{Dm} is contained in $\Hom(\mathscr{I}/\mathscr{I}^2,W)$. Let $A_W(G):=\coker(m\colon \mathscr{I}/\mathscr{I}^2\to \Hom(\mathscr{I}/\mathscr{I}^2, W))$ be the cokernel of the commutator map.
\end{defn}

\begin{prop}\label{Pcohvan}
  Let $W \subset \mathscr{I}^2/\mathscr{I}^3$ be as in Definition~\ref{Dnatsub}. Suppose that there exists an element $\sigma\in \Aut(G)$ which acts on $G^{\ab}$ as multiplication by\, $-1$. Then the class $4\MD_{{\univ}}$ lies in the image of the natural map
  $$H^1\left(\Aut(G),\Hom\left(\mathscr{I}/\mathscr{I}^2,W\right)\right) \lra H^1\left(\Aut(G),\Hom\left(\mathscr{I}/\mathscr{I}^2,\mathscr{I}^2/\mathscr{I}^3\right)\right).$$
  If
  $$H^0\left(\Aut(G), \Hom\left(\mathscr{I}/\mathscr{I}^2,\mathscr{I}^2/\mathscr{I}^3\right)\otimes \mathbb{Z}_\ell/2\right)=0,$$
  then $2\MD_{{\univ}}$ has a unique preimage  under this map.
\end{prop}

We will prove this proposition at the end of this subsection. Note that if $\ell\neq 2$, the proposition implies that $\MD_{\univ}$ itself is in the image of the map in question with a unique preimage.

An analogous result for $J_{\univ}$ will follow.

\begin{prop}\label{PJohnCoeff}
  Let $A_W(G)$ be the cokernel of the commutator map defined in Definition \ref{Dnatsub}.  Suppose that there exists an element $\sigma\in \Aut(G)$ which acts on $G^{\ab}$ as multiplication by\, $-1$.  Then the class $4J_{{\univ}}$ lies in the image of the natural map $H^1(\Out(G),A_W(G)) \rightarrow H^1(\Out(G),A(G))$. If
  $$H^0\left(\Out(G), \Hom\left(\mathscr{I}/\mathscr{I}^2,\mathscr{I}^2/\mathscr{I}^3\right)\otimes \mathbb{Z}_\ell/2\right)=0,$$
  the class $2J_{\univ}$ has a canonical preimage under this map, characterized as follows. Letting $\widetilde{\MD}_\univ$ be the unique preimage of $2\MD_\univ$ produced in Proposition \ref{Pcohvan}, the canonical preimage of $2J_\univ$ in $H^1(\Out(G), A_W(G))$ will be the unique element mapping to the image of\, $\widetilde{\MD}_\univ$ in $H^1(\Aut(G), A_W(G))$. 
\end{prop}

Before proving Proposition~\ref{Pcohvan}, we first need a lemma.

\begin{lemma}\label{Lassemble}
Let $G$ be a finitely generated pro-$\ell$ group with $\mathscr{I}/\mathscr{I}^2, \mathscr{I}^2/\mathscr{I}^3$ torsion-free. Let $V \colonequals \mathscr{I}/\mathscr{I}^2$, and let $W \subset \mathscr{I}^2/\mathscr{I}^3$ be as in Definition~\ref{Dnatsub}. Let $S$ be the image of the natural map $\Aut(G) \rightarrow \Aut(V)$, and let $T \colonequals \ker(\Aut(G) \rightarrow S)$. Then: 
\begin{enumerate}
\item\label{Step1} 
 Assume that $-\text{id}_V$  is in $S$. Then the group $H^i(S,\Hom(V,U))$ is $2$-torsion for any $\Aut(G)$-sub\-quo\-tient $U$ of $V \otimes V$ and any $i\in \mathbb{Z}_{\geq 0}$.
\item \label{Step1.5}
  Assume that $-\text{id}_V$  is in $S$. Then we have
  $$ H^1(S,M)\simeq H^0(S, M\otimes \mathbb{Z}_\ell/2)$$
  for any torsion-free $\Aut(G)$-subquotient $M$\! of\, $\Hom(V, V \otimes V)$.
\item\label{Step2} 
The image of the class 
$\MD_{{\univ}}$ under the restriction map
$$H^1\left(\Aut(G),\Hom\left(V,\mathscr{I}^2/\mathscr{I}^3\right)\right) \lra H^1\left(T,\Hom\left(V,\mathscr{I}^2/\mathscr{I}^3\right)\right)$$
lies in the image of the natural injective \footnote{The injectivity follows from the injectivity of $W \hookrightarrow \mathscr{I}^2/\mathscr{I}^3$ and the fact that $T$ acts trivially on $\Hom(V,W)$ and $\Hom(V,\mathscr{I}^2/\mathscr{I}^3)$.} map
$$H^1(T,\Hom(V,W)) \lra H^1\left(T,\Hom\left(V,\mathscr{I}^2/\mathscr{I}^3\right)\right).$$  
\end{enumerate}
\end{lemma}

\begin{remark}\label{Rtorsfreeab}\leavevmode
\begin{enumerate}
    \item The assumption in Lemma~\ref{Lassemble}\eqref{Step1} is satisfied by finitely generated free pro-$\ell$ groups and pro-$\ell$ surface groups (\textit{i.e.}, the pro-$\ell$ completion of the fundamental group of a genus $g$ Riemann surface). Indeed, Lemma~\ref{Lfreeprol}\eqref{Lfreeprol-2} implies that $S=\Aut(V)$ in the first case, and  \cite[Proposition~1]{AsKa} shows that $S \cong \GSp_{2g}(\Z_\ell)$ in the second case.
    \item By the above remark and direct computation, the hypothesis that
      $$H^0\left(\Aut(G), \Hom\left(\mathscr{I}/\mathscr{I}^2, \mathscr{I}^2/\mathscr{I}^3\right)\otimes \mathbb{Z}_\ell/2\right)=0$$
      in Propositions \ref{Pcohvan} and \ref{PJohnCoeff} is satisfied for finitely generated free pro-$\ell$ groups and for pro-$\ell$ surface groups. 
    \item Note that the statement of Lemma~\ref{Lassemble}\eqref{Step2} is a pro-$\ell$ version of Johnson's theorem \cite{Johnson} on the mapping class group of a Riemann surface with a marked point. 

\end{enumerate}
\end{remark}

\begin{proof}[Proof of Lemma~\ref{Lassemble}]
\eqref{Step1}~ The proof is the same as that of \cite[Lemma 5.4]{HM}.

\eqref{Step1.5}~ This is again similar to that of \cite[Lemma 5.4]{HM}; it is immediate from the Bockstein sequence associated to the short exact sequence
$$
0\lra M\overset{\cdot 2}{\lra} M\lra M\otimes \mathbb{Z}_\ell/2\lra 0.
$$

\eqref{Step2}~
We first prove the result in the case where $G$ is a finitely generated free pro-$\ell$ group. Then we will reduce to this case.

\emph{The case that $G$ is a finitely generated free pro-$\ell$ group, generated by $g_1, \ldots, g_r$.} 
Let
$$
\Delta \colon \Z_\ell[[G]] \lra \Z_\ell[[G]] \otimes \Z_\ell[[G]]
$$
denote the comultiplication map of the group ring $\Z_\ell[[G]]$, \textit{i.e.}, the map defined by
$$\Delta\colon g\longmapsto g\otimes g$$
for $g\in G$ and extended linearly.  
By Lemma \ref{Lfreeprol}\eqref{Lfreeprol-1}, the set
$$\left\{x_1,\ldots,x_r,x_1^2,x_1x_2,\ldots,x_rx_{r-1}, x_r^2\right\}$$
is a $\Z_\ell$-basis for $\mathscr{I}/\mathscr{I}^3$. As any $\sigma \in T$ preserves $\mathscr{I}$ and fixes $\mathscr{I}/\mathscr{I}^2$, there exist unique elements $b_i^{kl}(\sigma)\in\Z_\ell$ such that
\begin{equation}\label{Eautexp} \sigma(x_i) = x_i + \sum_{kl} b_i^{kl}(\sigma) x_k x_l \mod \mathscr{I}^3. \end{equation}

From the  commutative diagram
\begin{center}
    \begin{tikzcd}
\Z_\ell[[G]]  \arrow[r,"\Delta"] \arrow[d,"\sigma"] & \Z_\ell[[G]] \otimes \Z_\ell[[G]] \arrow[d,"\sigma \otimes \sigma"] \\
\Z_\ell[[G]] \arrow[r,"\Delta"] & \Z_\ell[[G]] \otimes \Z_\ell[[G]], 
\end{tikzcd}
\end{center}
for every $i$ we have 
\begin{equation}\label{Ecommute} \Delta(\sigma(x_i)) = (\sigma \otimes \sigma) (\Delta(x_i)). \end{equation}
We now compute both sides of this equality. Since $\Delta(g_i) = g_i \otimes g_i$ for all the generators $g_i$, we can compute that
\begin{equation}\label{Elieelements} \Delta(x_i) = \Delta(g_i-1) = (x_i+1) \otimes (x_i + 1) - 1 = x_i \otimes x_i + 1 \otimes x_i + x_i \otimes 1 \end{equation}
for the corresponding generators $x_i=g_i-1$ of the augmentation ideal $\mathscr{I}$. Since $\Delta$ is a ring homomorphism, we also have
\begin{equation}\label{Emultiplicative} \Delta(x_k x_l) = \Delta(x_k) \Delta(x_l) \end{equation}
for every pair of indices $k,l$. Combining \eqref{Eautexp}, \eqref{Elieelements}, \eqref{Emultiplicative} with \eqref{Ecommute} and comparing coefficients of $x_k x_l$ on both sides gives
\begin{align}\label{Ecocycskew}
    b_i^{kl}(\sigma) + b_i^{lk}(\sigma) = 0&\quad \text{if } k \neq l, \\
    2b_i^{kk}(\sigma) = 0 &\quad\text{if } k= l
\end{align}
or, equivalently by Definition~\ref{Dnatsub}, that 
\begin{equation}\label{Eskewsym} \sum_{kl} b_i^{kl}(\sigma) x_k x_l \in W \quad \text{ for every } i.
\end{equation}
Finally, explicit computation gives that
$$\MD_{\univ}|_T \in H^1\left(T,\Hom\left(\mathscr{I}/\mathscr{I}^2,\mathscr{I}^2/\mathscr{I}^3\right)\right)$$
is represented by the cocycle
\begin{equation}\label{Ecocyclerep}
  \sigma \longmapsto\left(x_i \longmapsto  \sum_{kl} b_i^{kl}(\sigma) x_k x_l\right) \mod \mathscr{I}^3.
\end{equation}
Combining this with \eqref{Eskewsym}, we get that the explicit cocycle~\eqref{Ecocyclerep} representing $\MD_{\univ}$ restricted to $T$ is visibly in the image of the map
\[ H^1(T,\Hom(V,W)) \lra H^1(T,\Hom(V,V\otimes V)). \qedhere \]

\emph{Reduction to the case that $G$ is free pro-$\ell$.} We now let $\tilde G$ be an arbitrary finitely generated pro-$\ell$ group with torsion-free abelianization. Let $G$ be a free pro-$\ell$ group, and view 
$$\pi\colon G\lra \tilde G$$
as a surjection inducing an isomorphism on abelianizations. Let $T_G\subset \Aut(G)$ be the subgroup consisting of automorphisms of $G$ which descend to automorphisms of $\tilde G$ and act trivially on $G^{\ab}$. Let $T_{\tilde G}\subset \Aut(\tilde G)$ be the subgroup acting trivially on $\tilde G^{\ab}$. By Lemma \ref{Lfreeprol}\eqref{Lfreeprol-3}, the natural map $T_G\to T_{\tilde G}$ is surjective. 

Since $T_{\tilde G}$ acts trivially on $\Hom(V, \mathscr{I}_{\tilde G}^2/\mathscr{I}_{\tilde G}^3)$, we may rewrite
$$H^1\left(T_{\tilde G}, \Hom\left(V, \mathscr{I}_{\tilde G}^2/\mathscr{I}_{\tilde G}^3\right)\right)=\Hom\left(T_{\tilde G}, \Hom\left(V, \mathscr{I}_{\tilde G}^2/\mathscr{I}_{\tilde G}^3\right)\right);$$
we wish to show that the homomorphism in question factors through $\Hom(V, W_{\tilde G})$. But this is immediate for the analogous fact for $G$, combined with the fact that $W_G$ surjects onto $W_{\tilde G}$, by definition.
\end{proof}

\begin{proof}[Proof of Proposition~\ref{Pcohvan}]
Let $V$, $S$, $T$ be as in Lemma~\ref{Lassemble}. Note that as $\mathscr{I}^2/\mathscr{I}^3$ is assumed to be torsion-free in Definition \ref{Dnatsub}, $W$ and hence $\Hom(V, W)$ are torsion-free.

Using the inflation-restriction sequence for the exact sequence of groups
\[ 0 \lra T \lra \Aut(G) \lra S \lra 0, \]
we get  the following commutative diagram of exact sequences:
$$\resizebox{\hsize}{!}{
\xymatrix@C=1em{
H^1\left(S,\Hom(V,W)^T\right) \ar[r]\ar[d] & H^1(\Aut(G),\Hom(V,W)) \ar[r]\ar[d] & H^1(T,\Hom(V,W))^S \ar[r] \ar@{^{(}->}[d] & H^2\left(S,\Hom(V,W)^T\right) \ar[d] \\
H^1\left(S,\Hom\left(V,\mathscr{I}^2/\mathscr{I}^3\right)^T\right) \ar[r] & H^1\left(\Aut(G),\Hom\left(V,\mathscr{I}^2/\mathscr{I}^3\right)\right) \ar[r] & H^1\left(T,\Hom\left(V,\mathscr{I}^2/\mathscr{I}^3\right)\right)^S\ar[r] & H^2\left(S, \Hom\left(V,\mathscr{I}^2/\mathscr{I}^3\right)^T\right).
}
}$$

Lemma~\ref{Lassemble}(\ref{Step1}) shows that $H^i(S,\Hom(V,W)^T)$ and $H^i(S,\Hom(V,\mathscr{I}^2/\mathscr{I}^3)^T)$ are 2-torsion. The image of $\MD_\univ$ in $H^1(T, \Hom(V, \mathscr{I}^2/\mathscr{I}^3))^S$ lies in the image of the natural map
$$H^1\left(T,\Hom(V,W)\right)^S\lra H^1\left(T, \Hom\left(V, \mathscr{I}^2/\mathscr{I}^3\right)\right)^S$$
by Lemma \ref{Lassemble}(\ref{Step2}); a diagram chase now yields the claim about $4\MD_\univ$. 

Now if
$$H^0\left(\Aut(G), \Hom\left(\mathscr{I}/\mathscr{I}^2,\mathscr{I}^2/\mathscr{I}^3\right)\otimes \mathbb{Z}_\ell/2\right)=0,$$
Lemma~\ref{Lassemble}(\ref{Step1.5}) implies that
$$H^1\left(S,\Hom(V,W)^T\right)=H^1\left(S,\Hom\left(V,\mathscr{I}^2/\mathscr{I}^3\right)^T\right)=0.$$
Now a diagram chase finishes the proof.  
\end{proof}

\begin{proof}[Proof of Proposition \ref{PJohnCoeff}]
The claim about $4J_\univ$ is immediate from Proposition \ref{Pcohvan}; we now explain the claim about $2J_\univ$. We have a commutative diagram
$$\xymatrix{
H^1\left(\Aut(G),\Hom\left(\mathscr{I}/\mathscr{I}^2,W\right)\right) \ar[r] \ar[d]&  H^1\left(\Aut(G),\Hom\left(\mathscr{I}/\mathscr{I}^2,\mathscr{I}^2/\mathscr{I}^3\right)\right) \ar[d]\\
H^1(\Aut(G),A_W(G)) \ar[r] &  H^1(\Aut(G),A(G))  \\
H^1(\Out(G),A_W(G))\ar[r] \ar[u]&  H^1(\Out(G),A(G)). \ar[u]
}$$
Letting $\widetilde{\MD}_\univ\in H^1(\Aut(G),\Hom(\mathscr{I}/\mathscr{I}^2,W))$ be the unique preimage of $2\MD_\univ$ produced by Proposition \ref{Pcohvan}, it suffices to show that the image of $\widetilde{\MD}_\univ$ in $H^1(\Aut(G),A_W(G))$ has a unique preimage in $H^1(\Out(G), A_W(G))$. But this follows analogously to the proof of Proposition \ref{prop:Juniv}.
\end{proof}

As a consequence of Remark \ref{Rtorsfreeab}, we have the following. 

\begin{cor}\label{cor:uniquepreimage}
  Suppose $G$ is a finitely generated free pro-$\ell$ group or a pro-$\ell$ surface group. Then $2\MD_{\univ}$  has a unique preimage $\widetilde{\MD}$ under the natural map
  $$H^1\left(\Aut(G),\Hom\left(\mathscr{I}/\mathscr{I}^2,W\right)\right) \lra H^1\left(\Aut(G),\Hom\left(\mathscr{I}/\mathscr{I}^2,\mathscr{I}^2/\mathscr{I}^3\right)\right).$$
  Moreover, $2J_{\univ}$ has a canonical preimage $\widetilde J$ under the natural map
  $$H^1(\Out(G),A_W(G))\lra H^1(\Out(G), A(G)).$$
\end{cor}

\subsection{Ceresa classes of curves in $\boldsymbol{\ell}$-adic cohomology}\label{SCerFunGrp}
Let $X$ be a curve over $K$, and let $\ell$ be a prime different from the characteristic of $K$. For $\bar x$ a geometric point of $X$,  let 
$$o_\ell\colon \textup{Gal}\left(\bar K/K\right)\lra \Out\left(\pi_1^{\ell}(X_{\bar K}, \bar x)\right)$$
be the map coming from the natural outer action of $\textup{Gal}(\bar K/K)$ on $\pi_1^{\text{\'et}}(X_{\bar K}, \bar x)$; here $\pi_1^{\ell}(X_{\bar K}, \bar x)$ is the pro-$\ell$ completion of $\pi_1^{\text{\'et}}(X_{\bar K}, \bar x)$. Note that $\Out(\pi_1^{\ell}(X_{\bar K}, \bar x))$ is independent of $\bar x$. If $y\in X(K)$ is a rational point and $\bar y$ the geometric point obtained by some choice of algebraic closure $K \hookrightarrow \bar{K}$, we let
$$a_{\ell, y}\colon \textup{Gal}\left(\bar K/K\right)\lra \Aut\left(\pi_1^{\ell}(X_{\bar K}, \bar y)\right)$$
be the map induced by the canonical Galois action on $\pi_1^{\text{\'et}}(X_{\overline{K}}, \bar y)$.

\begin{defn}\label{DladicCeresa}
The \emph{modified diagonal class $\MD(X,\bar y)$ of the pointed curve $(X,\bar y)$} is the pullback $a_{\ell,y}^*\MD_{\univ}$ of the group-theoretic modified diagonal class $\MD_{\univ}$ for the group $\pi_1^{\ell}(X_{\bar K}, \bar y)$ defined in Definition~\ref{DuniC}; it depends on the choice of the rational base point $y$.
 The \emph{Johnson class $J(X)$ of the curve $X$} is the pullback $o_\ell^*J_{\univ}$ of the group-theoretic Johnson class $J_{\univ}$ 
 for the group $\pi_1^{\ell}(X_{\bar K}, \bar x)$ defined in Definition~\ref{DJuniv}; it is by definition independent of the choice of geometric point $\bar x$. 
\end{defn}

\begin{remark}\label{rem:tilde-classes}
Similarly, one may define classes $\widetilde{\MD}(X, b)$ and $\widetilde{J}(X)$ by pulling back the classes $\widetilde{\MD}$ and $\widetilde{J}$ of Corollary \ref{cor:uniquepreimage}. Note that in general some $2$-torsion information is lost when passing from $\MD$ to $\widetilde{\MD}$ (resp.~$J$ to $\widetilde{J}$).
\end{remark}

\subsubsection{Comparison to the Ceresa classes in \cite{HM}}\label{projective-subsubsection}
For the rest of Section~\ref{SCerFunGrp}, 
we consider the case where $X$ is a smooth, projective, and geometrically integral curve of genus $g$ over a field $K$, with a rational point $b \in X(K)$. We let $G$ be the pro-$\ell$ \'etale fundamental group $\pi_1^\ell(X \otimes \overline{K}, \overline{b})$ and let $\mathscr{I}$ be the augmentation ideal in $\mathbb{Z}_\ell[[G]]$, as before. 
The purpose of this paragraph
is to compare the classes $\MD(X,b)$ and $J(X)$ to the classes $\mu(X,b)$ and $\nu(X)$ defined in \cite{HM} arising from the Ceresa cycle. Explicitly, we show $\widetilde{\MD}(X,b) = \mu(X,b)$ and $\widetilde{J}(X) = \nu(X)$. For a comparison between the extension classes of mixed Hodge structures arising from the modified diagonal cycle and the Ceresa cycle, see \cite[Section 1]{DarmonRotgerSols}.

\begin{lemma}\label{LPoincaredu}
  There are canonical isomorphisms of Galois-modules
  \begin{equation}\label{Ealliso} \mathscr{I}/\mathscr{I}^2 \simeq G^{\ab} \simeq H^1_{\textup{\'{e}t}}(X_{\bar K},\mathbb{Z}_\ell)^\vee.
\end{equation}
\end{lemma}

\begin{proof}
See Proposition~\ref{PGabaug} for the first isomorphism and \cite[Example~11.3]{Milne} for the second isomorphism. \qedhere
\end{proof}

\begin{lemma}\label{LImodI2}
  Let $H \colonequals \mathscr{I}/\mathscr{I}^2$, and let
  $$\omega\colon\mathbb{Z}_\ell(1)\lra H^{\otimes 2}$$
  be the map dual to the cup product $$H^1(X_{\overline{K}}, \mathbb{Z}_\ell)\otimes H^1(X_{\overline{K}}, \mathbb{Z}_\ell)\lra H^2(X_{\overline{K}}, \mathbb{Z}_\ell) \simeq \mathbb{Z}_\ell(-1)$$
under the identification from Lemma~\ref{LPoincaredu}.
Then we have an exact sequence
\[ 0 \lra \Z_\ell(1) \overset{\omega}{\lra} H^{\otimes 2} \lra \mathscr{I}^2/\mathscr{I}^3 \lra 0, \]
where the rightmost map is the natural multiplication map. 
\end{lemma}

\begin{proof}
This is presumably well known; we give a sketch of how to deduce it from existing literature. The analogous theorem for compact Riemann surfaces is immediate from \cite[Corollary 8.2]{Hain02}. Now the result follows by taking pro-$\ell$ completions of the sequence in \cite[Corollary 8.2]{Hain02} and comparing (1) the pro-$\ell$ completion of the group ring of a Riemann surface to $\mathbb{Z}_\ell[[G]]$ and (2) the singular cohomology of a compact Riemann surface to the $\ell$-adic cohomology of $X_{\overline{K}}$. (Strictly speaking, the comparison above goes as follows: if necessary, lift $X$ to characteristic zero. Then spread out, embed the ground ring in $\mathbb{C}$, and analytify. These arguments are lengthy and standard, so we omit them.)
\end{proof}

Recall from Definition~\ref{Dnatsub} that $W \subset \mathscr{I}^2/\mathscr{I}^3$ is  the image of $\textup{Alt}^2H \subset H^{\otimes 2}$ under the multiplication map
$H^{\otimes 2} \rightarrow \mathscr{I}^2/\mathscr{I}^3. $

\begin{lemma}\label{Lcorcoeff} 
  Restricting the multiplication map $H^{\otimes 2}\to \mathscr{I}^2/\mathscr{I}^3$ to $\textup{Alt}^2 H$ induces an isomorphism
  $$(\textup{Alt}^2 H)/\textup{Im}(\omega)\isor W.$$
\end{lemma}

\begin{proof}
It suffices to show that the map $\omega$ of Lemma \ref{LImodI2} factors through $\textup{Alt}^2 H$. But this is immediate from the fact that the cup product on $H^1(X_{\overline{K}}, \mathbb{Z}_\ell)$ is alternating. 
\end{proof}

In \cite[Sections 5 and 10]{HM}, Hain and Matsumoto define classes $m(X,b)$ and $n(X)$ in Galois cohomology, which control the action of the absolute Galois group of $K$ on the quotient of $\pi_1^\ell(X_{\overline{K}}, b)$  by the second piece of the lower central series. In \cite[Theorem 3 and Section~10.5]{HM}, 
they compare these classes to classes $\mu(X,b)$ and $\nu(X)$ arising from the Ceresa cycle under the cycle class map. We briefly compare our classes to theirs when $X$ is smooth and proper.

\begin{prop}\label{Lcompare}
Let $\mu(X,b)$ and $\nu(X)$ be the classes in \cite[Section 4]{HM} constructed from the image of the Ceresa cycle under a cycle class map. Recall from Remark \ref{rem:tilde-classes} the classes $\widetilde{\MD}(X,b)$, $\widetilde{J}(X)$ constructed from $2\MD(X, b)$, $2J(X)$. Then $\widetilde{\MD}(X,b) = \mu(X,b)$ and ${\widetilde J}(X)=\nu(X)$.
\end{prop}

\begin{proof}
We give a sketch for $\widetilde{\MD}(X, b)$; the case of $\widetilde{J}(X)$ follows analogously. Let
\[ G = L^1G \supset L^2G \supset \cdots,\quad \textup{where } L^{k+1}G = \overline{[G,L^kG]},
\]
be the lower central series filtration of $G$. By \cite[Corollary 4.2]{Quillen}, we have the following commutative diagram of exact sequences, where all maps are compatible with the induced $\Aut(G)$-actions:  
\begin{center}
\begin{tikzcd}
0 \arrow[r]& L^2G/L^3G \arrow[r]\arrow[d,hook] & G/L^3G \arrow[r]\arrow[d,hook] & H \arrow[r]\arrow[d,"\simeq"]& 0\\
0 \arrow[r] & \mathscr{I}^2/\mathscr{I}^3 \arrow[r] & \mathscr{I}/\mathscr{I}^3 \arrow[r]& \mathscr{I}/\mathscr{I}^2 \arrow[r]& 0.
\end{tikzcd}
\end{center}
Here all the vertical maps are induced by sending a group element $g$ to $g-1$. Note that the middle vertical inclusion is only a set-theoretic map, not a homomorphism.

Let
$$s\colon H\lra G/L^3G,\quad v \longmapsto s(v)$$
be a set-theoretic section to the quotient map $G/L^3G\to H$, and let 
$$s'\colon \mathscr{I}/\mathscr{I}^2\lra \mathscr{I}^2/\mathscr{I}^3,\quad v-1 \longmapsto s(v)-1$$ 
be the induced map. Let $T\subset \Aut(G)$ be the subgroup acting trivially on $H$. From the top sequence, following \cite[Section 5.1]{HM}, we get the Magnus homomorphism $\tilde\epsilon\mkern-1mu\in\mkern-1mu \Hom(T, \Hom(H,L^2G/L^3G))^{\GSp(H)}$:
\[\tilde\epsilon: g \longmapsto \left(v \longmapsto g(s(v))s(v)^{-1} \bmod L^3G\right). \]

By \cite[Proposition 5.5]{HM}, there is a unique class $m \in H^1(\Aut G,\Hom(H,L^2G/L^3G) )$ whose image under 
\[ H^1\left(\Aut G,\Hom\left(H,L^2G/L^3G\right)\right) \lra H^0\left(\GSp H, H^1\left(T, \Hom\left(H,L^2G/L^3G\right)\right)\right) \]
is $2\tilde\epsilon$ under the canonical identification 
$$\Hom\left(T, \Hom\left(H, L^2G/L^3G\right)\right)^{\GSp H}\simeq H^0\left(\GSp H, H^1\left(T, \Hom\left(H,L^2G/L^3G\right)\right)\right).$$
By \cite[Theorem 3]{HM}, the pullback of $m$ by $\operatorname{Gal}(\bar{K}/K) \to \Aut G$ induced by $b$ agrees with the Ceresa class $\mu(X,b)$.

Now let us rewrite $g(s(v))s(v)^{-1}-1$ modulo $\mathscr{I}^3$:
\begin{align*}
   g(s(v))s(v)^{-1} -1  & = g(s(v))(s(v)^{-1}-1)+g(s(v)) -1\\
    & \equiv g(s(v))( 1-s(v) +(1-s(v))^2) +g(s(v)) -1\\ 
    & = g(s(v))( 1-s(v)) +(g(s(v))-1)(1-s(v))^2+(1-s(v))^2 +g(s(v)) -1\\
    & \equiv g(s(v))( 1-s(v))+(1-s(v))^2 +g(s(v)) -1\\
    &=g(s(v)) -s(v) + (g(s(v))-s(v))(1-s(v))\\
    &\equiv g(s(v))-s(v).
\end{align*}
Here we use the substitution $$s(v)^{-1}-1=1-s(v) +(1-s(v))^2 \bmod \mathscr{I}^3$$ and the fact that $$(g(s(v))-1)(1-s(v))^2, (g(s(v))-s(v))(1-s(v))\in\mathscr{I}^3$$ (because $g(s(v))-s(v)\in \mathscr{I}^2$ by the definition of $T$).

But the cocycle representing the $\MD_{\univ}|_T$ is
\[ g \longmapsto ( v-1 \longmapsto g(s(v)-1)-(s(v)-1)   =g(s(v))-s(v)), \]
which proves that the two classes are the same under restriction to $T$. Now comparing the diagram chases in the proof of Proposition \ref{Pcohvan} and \cite[Proposition 5.5]{HM} (using the identification from Lemma \ref{Lcorcoeff}) completes the proof.
\end{proof}

\subsubsection{Stability under base change}
We finally observe that the property of the Johnson or modified diagonal class being torsion is in fact a geometric property---that is, it descends through finite extensions of the ground field.

\begin{prop}\label{prop:torsion-geometric}
Let $K$ be a field and $X$ a smooth, geometrically connected curve over $K$. Let $\ell$ be a prime different from the characteristic of $K$ and $J(X)$ the associated Johnson class; if\, $b\in X(K)$ is a rational point, we let $\MD(X, b)$ be the modified diagonal class. Let $L/K$ be a finite extension. Then $J(X_L)$ $($resp.~$\MD(X_L, b_L))$ is torsion if and only if\, $J(X)$ $($resp.~$\MD(X, b))$ is torsion.
\end{prop}

\begin{proof}
Choose an algebraic closure $\overline{K}$ of $K$ (and hence of $L$). Let $i_{L/K}\colon \textup{Gal}(\overline{K}/L)\to \textup{Gal}(\overline{K}/K)$ be the natural map; then it follows from the definition that $i_{L/K}^*J(X)=J(X_L)$ (resp.~$i_{L/K}^*\MD(X, b)=\MD(X_L, b_L)$). This proves the ``if'' direction.

To see the ``only if'' direction, suppose $J(X_L)$ (resp.~$\MD(X_L, b_L)$) is torsion. Then $i_{L/K *}i_{L/K}^*J(X)$ (resp.~$i_{L/K *}i_{L/K}^*\MD(X,b)$) is torsion (where here $i_{L/K*}$ denotes the corestriction map). But $i_{L/K *}i_{L/K}^*$ is simply multiplication by the index of $\textup{Gal}(\overline{K}/L)$ in $\textup{Gal}(\overline{K}/K)$, which completes the proof.
\end{proof}

\section{Curves with torsion modified diagonal or Johnson class}\label{Storsion}

\subsection{$\boldsymbol{\Aut(X)}$-invariance}
Let $X$ be a smooth geometrically connected curve over a field $K$ and $\ell$ a prime different from the characteristic of $K$. Choose a geometric point $\overline{x}$ of $X$, and let $G=\pi_1^\ell(X_{\overline{K}}, \overline{x})$.

In this subsection, we show that $\Aut_K(X)$ places restrictions on the Johnson class $J(X)$; analogously, $\Aut_K(X, b)$ places restrictions on $\MD(X, b)$ for $b\in X(K)$.

\begin{prop}\label{AutProp}
Let $B\subset \Aut_K(X)$ be  a finite subgroup such that $H^0(B, A(G))=0$. Then the Johnson class $J(X)$ is torsion with order $d \mid \# B$. 
Likewise, for $b\in X(K)$, if $B'\subset \Aut_K(X, b)$ is a finite subgroup with $H^0(B', \Hom(\mathscr{I}/\mathscr{I}^2,\mathscr{I}^2/\mathscr{I}^3))=0$, then the class $\MD(X,b)$ is torsion with order $d \mid \# B'$.
\end{prop}

\begin{proof}
We  first prove the statement for $J(X)$.

We apply the inflation-restriction sequence to the group extension
$$1\lra B \lra \Gal\left(\bar{K}/K\right) \times B \lra \Gal\left(\bar{K}/K\right) \lra 1,$$
which gives
 \begin{align*}\label{AutCsequence}
     0\lra H^1\left(\Gal\left(\bar{K}/K\right), A(G)^{B}\right)&\lra H^1\left(\Gal\left(\bar{K}/K\right) \times B, A(G) \right)
     \lra H^1\left(B, A(G)\right)^{\Gal\left(\bar{K}/K\right)}.
 \end{align*}

Since $B$ is a finite group, its cohomology $H^n(B,M)$ has exponent dividing $\# B$ for any finitely generated $\mathbb{Z}_\ell[B]$-module $M$ and any $n>0$.
The pullback of the Johnson class $J_{\univ}$ via
$$B \times \Gal\left(\bar{K}/K\right) \lra \Out\left(\pi_1^\ell\left(X_{\bar{K}}\right)\right)$$
lives in $ H^1\!(\Gal(\bar{K}/K)) \times B, A(G) )$. So multiplying this class by $\!\# B$ gives a class in $H^1\!(\Gal(\bar{K}/K)), A(G)^{B})$. But by assumption, $A(G)^{B}=0$, and we conclude.

The proof is the same for the class $\MD(X,b)$ with
the coefficients
$A(G)$ replaced by $\Hom(\mathscr{I}/\mathscr{I}^2,\mathscr{I}^2/\mathscr{I}^3)$ and $B$ replaced by $B'$.
\end{proof}

\subsection{Hyperelliptic curves}

\begin{prop}\label{hyperelliptic}
When $X$ is a hyperelliptic curve, the class $J(X)$ is $2$-torsion. Moreover, if\, $X$ has a rational Weierstrass point $x$, the class $\MD(X,x)$ is also $2$-torsion.
\end{prop}

\begin{proof}
Let $\iota \in \Aut_K(X)$ denote the hyperelliptic involution on $X$. Then $\iota$ acts on $H_1(X,\mathbb{Z}) \cong \mathscr{I}/\mathscr{I}^2$ as multiplication by $-1$, and hence on $\mathscr{I}^2/\mathscr{I}^3$ as the identity. Thus $ \Hom(\mathscr{I}/\mathscr{I}^2,\mathscr{I}^2/\mathscr{I}^3)^\iota=0$. Now the statements follow from Proposition~\ref{AutProp}, applied with $B=B'=\langle \iota \rangle$.
\end{proof}

\begin{remark}
The method used in Proposition~\ref{hyperelliptic} cannot yield similar results for superelliptic curves, using the cyclic group $\Aut (C/\mathbb{P}^1)$, as we now explain. For a degree $n$ cyclic cover of the projective line, pick a prime $p \mid n$ so that  we have $\mu_p \subset \Aut (C/\mathbb{P}^1)$ (here $\mu_p$ is the set of $\supth{p}$ roots of unity). Given a primitive root of unity $\zeta_p \in \mu_p$, its action on $H=H^1_{\mathrm{sing}}(C,\mathbb{C})$ gives a decomposition $H = \oplus^{p-1}_{i=1} V_i$ where $\zeta_p$ acts on $V_i$ as multiplication by $\zeta_p^i$. Then we have $\dim V_i= {2g}/({p-1})$, which in particular does not depend on $i$; see \cite{Moonen}. Similarly, $H \otimes H$ also decomposes into eigenspaces for the $\zeta_p$-action, and all the $V_i$ for $i =1 ,\ldots, p-1$ appear with non-zero multiplicity in this decomposition. Therefore, we cannot rule out non-trivial $\Aut (C/\mathbb{P}^1)$-equivariant maps between $H$ and $H \otimes H$ using this isotypic decomposition alone.
\end{remark}

\subsection{The  Fricke--Macbeath curve}\label{sec:Fricke-Macbeath}

The Fricke--Macbeath curve $C$ is the unique Hurwitz curve over $\overline{\mathbb{Q}}$ of genus $7$. Its automorphism group is the simple group $\PSL_2(8)$ of order $504$; see \cite[p.~541]{MR0177342}. The simplicity of $\textup{PSL}_2(8)$ implies that there is no central order $2$ element in $\Aut_{\overline{\mathbb{Q}}}(C)$ and, in particular, $C$ is not hyperelliptic. By analyzing the action of the automorphism group on the homology of curve, we show the following.

\begin{prop}\label{FrickeMacbeath}
Let $X/K$ be a curve over a number field with $X_{\overline{\mathbb{Q}}}$ isomorphic to the Fricke--Macbeath curve $C$ above. The class $J(X)$ is torsion. 
\end{prop}

\begin{proof}
By Proposition~\ref{prop:torsion-geometric}, we may without loss of generality assume $\Aut_K(X) \cong \PSL_2(8)$, by replacing~$K$ with a finite extension.

We now choose an embedding $K\hookrightarrow \mathbb{C}$ and analyze the induced representation $\rho$ of $\Aut_K(X) \cong \PSL_2(8)$ on $H^1_{\text{sing}}(X(\mathbb{C})^{\text{an}}, \mathbb{Q})$. By standard comparison results, the representation of $\Aut_K(X)$ on $H^1(X_{\overline{\mathbb{Q}}, \text{\'et}}, \mathbb{Q}_\ell)$ will be isomorphic to the representation  obtained from $\rho$ by extending scalars from $\mathbb{Q}$ to $\mathbb{Q}_\ell$.

Hodge theory tells us  $H^1_{\mathrm{sing}}(C,\mathbb{C})$ decomposes as the direct sum of two complex-conjugate  $7$-di\-men\-sio\-nal $\PSL_2(8)$-representations $\chi, \overline{\chi}$. 
As $\PSL_2(8)$ in fact acts on $H^1_{\mathrm{sing}}(C,\mathbb{Q})$, it follows that the action of every element of $\PSL_2(8)$ on $H^1_{\mathrm{sing}}(C,\mathbb{C})$ has trace in $\mathbb{Q}$. Furthermore,  the action of $\PSL_2(8)$ on $H^1_{\mathrm{sing}}(C,\mathbb{C})$ is faithful since the genus of $X$ is greater than $1$. 

We now decompose $H_{\mathrm{sing}}^1(C,\mathbb{C}) = \chi \oplus \overline{\chi}$ as an $\Aut_K(X) \cong \textup{PSL}_2(8)$-representation using character theory. In the following table, $\zeta_n$ is a choice of primitive $\supth{n}$ root of unity and $\bar{\zeta_n}$ its complex conjugate.

\begin{small}
\begin{table}[!h]
\centering
\begin{tabular}{| c | c | c | c | c | c | c | c | c | c|}
\hline
Class		& 1 & 2		& 3		& 4				& 5				& 6				& 7				& 8				& 9				\\
Size		& 1	& 63 	& 56	& 72			& 72 			& 72 			& 56 			& 56 			& 56			\\
Order		& 1 & 2		& 3		& 7				& 7				& 7				& 9				& 9				& 9				\\
\hline
$\chi_1$ 	& 1	& 1		& 1 	& 1				& 1				& 1 			& 1 			& 1 			& 1	 			\\
$\chi_2$	& 7 & -1	& -2	& 0				& 0 			& 0				& 1				& 1				& 1				\\
$\chi_3$	& 7 & -1	& 1		& 0				& 0 			& 0				& $-\zeta_9-\bar{\zeta_9}$	& $\zeta_9^2+\bar{\zeta_9^2}$	& $\zeta_9^4+\bar{\zeta_9^4}$	\\
$\chi_4$	& 7 & -1	& 1		& 0				& 0 			& 0				& $\zeta_9^4+\bar{\zeta_9^4}$	& $-\zeta_9-\bar{\zeta_9}$		& $\zeta_9^2+\bar{\zeta_9^2}$	\\
$\chi_5$	& 7 & -1	& 1		& 0				& 0 			& 0			& $\zeta_9^2+\bar{\zeta_9^2}$	& $\zeta_9^4+\bar{\zeta_9^4}$	& $-\zeta_9-\bar{\zeta_9}$		\\
$\chi_6$	& 8 & 0		& -1	& 1				& 1 			& 1				& -1			& -1			& -1			\\
$\chi_7$	& 9 & 1		& 0		& $\zeta_7+\bar{\zeta_7}$		& $\zeta_7^2+\bar{\zeta_7^2}$	& $\zeta_7^3+\bar{\zeta_7^3}$	& 0				& 0				& 0				\\
$\chi_8$	& 9 & 1		& 0		& $\zeta_7^3+\bar{\zeta_7^3}$	& $\zeta_7+\bar{\zeta_7}$ 	& $\zeta_7^2+\bar{\zeta_7^2}$	& 0				& 0				& 0				\\
$\chi_9$	& 9 & 1		& 0		& $\zeta_7^2+\bar{\zeta_7^2}$	& $\zeta_7^3+\bar{\zeta_7^3}$	& $\zeta_7+\bar{\zeta_7}$	& 0				& 0				& 0				\\
\hline
\end{tabular}
\caption{Character table for $\textup{PSL}_2(8)$. See \cite[p.\ 6]{MR827219}.}
\label{tbl:PSL28CharTable}
\end{table}
\end{small}

First note that if the $7$-dimensional representation $\chi$ has a trivial subrepresentation, then this forces $\chi$ itself to be trivial (since the smallest non-trivial irreducible representation of $\PSL_2(8)$ has dimension~$7$). If this happens, then $\chi$, and in turn $\overline{\chi}$, are trivial $\PSL_2(8)$-representations. This contradicts the faithfulness of $H_{\mathrm{sing}}^1(C,\mathbb{C}) = \chi \oplus \overline{\chi}$ as a $\PSL_2(8)$-representation; hence $\chi$ is irreducible. So $H_{\mathrm{sing}}^1(C,\mathbb{C})$ decomposes as a sum of an irreducible $7$-dimensional representation and its complex conjugate.

Of the four $7$-dimensional irreducible representations $\chi_i, i=2,\ldots, 5$,
of $\PSL_2(8)$ in the character table below, the only one that has the property that $\chi \oplus \overline{\chi}$ has all its traces in $\mathbb{Q}$ 
is $\chi_2$. Hence
$$\rho \cong \chi_2\oplus \overline{\chi_2}\cong \chi_2\oplus \chi_2.$$

Now we compute the inner product
$$\langle \chi_2 \otimes \chi_2,\chi_2 \rangle = 7 \cdot 49-63-2 \cdot 4 \cdot 56+56+56+56 = 0.$$

Thus $\chi_2$ does not appear in the decomposition of $\chi_2 \otimes \chi_2$ into irreducibles. Hence there can be no $\PSL_2(8)$-equivariant map from $\chi_2 \oplus \chi_2$ to $(\chi_2 \oplus \chi_2)^{\otimes 2}$, which means 
$H^0(\PSL_2(8), \Hom(\mathscr{I}/\mathscr{I}^2,\mathscr{I}^2/\mathscr{I}^3)) =0$. Thus $H^0(\PSL_2(8), A(G))=0$, and by Proposition~\ref{AutProp}, the class $J(C)$ is torsion. Indeed, if $\Aut_K(C)=\PSL_2(8)$, then the class has order a divisor of $504$.
\end{proof}

\begin{cor}
Let $X$ be as in Proposition~\ref{FrickeMacbeath}. Then the Ceresa class $\nu(X)$ as defined in \cite{HM} is torsion.
\end{cor}

\begin{proof}
This is immediate from Proposition~\ref{Lcompare}.
\end{proof}

\begin{remark}
This is, to the authors' knowledge, the first known example of a non-hyperelliptic curve such that the image of the Ceresa cycle under the ($\ell$-adic) Abel--Jacobi map is torsion. An analogous argument (with the mixed Hodge structure on the Betti fundamental group) shows that the Hodge-theoretic analogue is also torsion (that is, the image of the Ceresa cycle in the appropriate intermediate Jacobian is torsion). It is natural to ask if the Ceresa cycle itself is torsion in the Chow ring of the Jacobian of $X$ modulo algebraic equivalence. Benedict Gross has explained to us that this is a prediction of the Beilinson conjectures.

It would be interesting to find (or prove the non-existence of) a positive-dimensional family of non-hyperelliptic curves with torsion Ceresa class.
\end{remark}

\subsection{Curves dominated by a curve with torsion modified diagonal or Johnson class}\label{sec:dominated-torsion}
In this last subsection, we prove the following. 

\begin{thm}\label{thm:dominated-torsion}
Let $X$ be a curve over a finitely generated field $k$ of characteristic zero, and let $f\colon X\to Y$ be a dominant map of curves over $k$. Then:
\begin{enumerate}
    \item\label{thm:dt-1} If\, $x\in X(k)$ is a rational point and $MD(X, x)$ is torsion, then $MD(Y, f(x))$ is torsion.
    \item\label{thm:dt-2} If\, $J(X)$ is torsion, then $J(Y)$ is torsion.
\end{enumerate}
\end{thm}

We view this as analogous to the fact that any curve dominated by a hyperelliptic curve is hyperelliptic.

As a corollary, we have the following. 

\begin{cor}\label{genus-3-example}
Let $\iota\in \PSL_2(8)$ be any element of order $2$. If\, $X/K$ is a curve of genus $7$ over a number field with $\textup{Aut}_K(X)\simeq \PSL_2(8)$, then $X/\langle \iota\rangle$ is a non-hyperelliptic curve of genus $3$ with $J(X/\langle\iota\rangle)$ torsion.
\end{cor}

\begin{remark}
Note that curves $X$ as above exist---for any model of the Fricke--Macbeath curve over a number field $K$, the base change to a finite extension of $K$ over which all the automorphisms are defined will suffice.
\end{remark}

\begin{proof}[Proof of Corollary~\ref{genus-3-example}]
The statement that $J(X/\langle\iota\rangle)$ is torsion is immediate from Theorem~\ref{thm:dominated-torsion} and Proposition~\ref{FrickeMacbeath}. So we need only verify that such curves have genus $3$ and are not hyperelliptic.

To see that $X/\langle \iota \rangle$ has genus $3$, note that $\PSL_2(8)$ has a unique conjugacy class of order $2$, whose trace (by the discussion in the proof of Proposition~\ref{FrickeMacbeath}) on $H^1(X)$ is $-2$. Hence by the Lefschetz fixed point theorem, $\iota$ has four fixed points. Now the Riemann--Hurwitz theorem gives the claim.

To show that $X/\langle \iota\rangle$ is not hyperelliptic, first we note that $\PSL_2(8)$ has a unique conjugacy class of order~$2$. Hence for any two elements $\iota_1, \iota_2$ in this conjugacy class, the quotient curves $X/\langle \iota_1 \rangle$ and $X/\langle \iota_2 \rangle$ are isomorphic. Now in \cite[Section 2]{TopVerschoor}, the authors give a model for one of the quotient curves---it is a smooth quartic curve in $\mathbb{P}^2$. Thus this isomorphism class of curves is non-hyperelliptic.
\end{proof}

We now give the proof of Theorem~\ref{thm:dominated-torsion}. We require the following lemmas.

\begin{lemma}\label{lem:unipotent-radical}
  Let $G$ be a group, and let
  $$0\lra U\lra V\to W\lra 0$$
  be an extension of\, $G$-representations over an algebraically closed field of characteristic zero, with $U, W$ semisimple. Then the extension splits if and only if the unipotent radical of the Zariski-closure of\, $G$ in $\GL(V)$ is trivial.
\end{lemma}

\begin{proof}
If the extension splits, then $V$ is semisimple. Hence the Zariski-closure of the image of $G$ is reductive, and we are done. 

On the other hand, assume the sequence does not split. We may without loss of generality replace $G$ with the Zariski-closure of its image in $\GL(V)$; we now wish to argue that $G$ is not reductive. Let $H\subset G$ be the kernel of the natural map $G\to \GL(U\oplus W)$; the subgroup $H$ is evidently unipotent and normal, so it suffices to show that $H$ is non-trivial. By the semisimplicity of $U\oplus W$, it follows that $G/H$ is reductive; hence inflation-restriction shows that $H^1(G, \Hom(W, U))\to H^1(H, \Hom(W, U))$ is injective (using the assumption of characteristic zero). But $H^1(G, \Hom(W,U))$ is non-trivial by assumption. Hence the same is true for $H^1(H, \Hom(W,U))$, and thus $H$ is non-trivial, as desired.
\end{proof}

\begin{lemma}\label{lem:extn-unipotent-radical}
  Let $G$ be a group, and let
  $$0\lra U\lra V\lra W\lra 0$$
  be an extension of\, $G$-representations over an algebraically closed field $k$ of characteristic zero, with $U, W$ semisimple. Let $S\subset G$ be a subgroup acting trivially on $U, W$, and let $m\colon S\to \Hom(W, U)$ be the induced map.  Then the image of the extension class of this sequence under the natural map
  $$H^1(G, \Hom(W, U))\lra H^1(G, \Hom(W,U)/\textup{im}(m))$$
  vanishes if and only if the unipotent radical of the Zariski-closure of\, $G$ in $\GL(V)$ equals the Zariski-closure of the image of\, $S$ in $\GL(V)$.
\end{lemma}

\begin{proof}
  Let $G$ be an abstract group and $\overline{G}$ an algebraic group over $k$, and let $G\to\overline{G}$ be a homomorphism with Zariski-dense image. Then a short exact sequence of $\overline{G}$-representations splits if and only if it splits as a sequence of $G$-representations; in other words, for any $\overline{G}$-representation $Q$, the natural map
  $$H^1\left(\overline{G}, Q\right)\longrightarrow H^1(G, Q)$$
  is injective. Thus we may without loss of generality replace $G$ with the Zariski-closure of its image in $\GL(V)$ and $S$ with the Zariski-closure of its image in the statement of the lemma. We make this replacement now.

Let $N\subset G$ be the kernel of the natural representation $G\to \GL(U\oplus W)$; this is a unipotent normal subgroup with reductive quotient (by the assumption that $U,W$ are semisimple) and hence equals the unipotent radical of $G$. By definition, we have $S\subset N$. We wish to show that the given vanishing holds in $H^1(G, \Hom(W, U)/\textup{im}(m))$ if and only if $S=N$.

Consider the short exact sequence
$$0\lra \Hom(W, U)/\textup{im}(m)\lra V' \lra k\lra 0$$
induced by our element of $H^1(G, \Hom(W, U)/\textup{im}(m))$. Then by definition, the kernel of $N \to \GL(V')$ is exactly $S$. Thus by Lemma~\ref{lem:unipotent-radical}, this extension splits if and only if $N \subset S$. This completes the proof.
\end{proof}

\begin{proof}[Proof of Theorem~\ref{thm:dominated-torsion}]
We first prove~\eqref{thm:dt-1}. Let $\mathscr{I}_X$ be the augmentation ideal in $\mathbb{Z}_\ell[[\pi_1^\ell(X_{\overline{k}}, \overline x)]]$, and let $\mathscr{I}_Y$ be the augmentation ideal in $\mathbb{Z}_\ell[[\pi_1^\ell(Y_{\overline{k}}, {f(\bar x)})]]$.

Let $U_X=\mathscr{I}_X^2/\mathscr{I}_X^3\otimes \mathbb{Q}_\ell$, $V_X=\mathscr{I}_X/\mathscr{I}_X^3\otimes \mathbb{Q}_\ell$, and $W_X=\mathscr{I}_X/\mathscr{I}_X^2\otimes \mathbb{Q}_\ell$, and similarly let $U_Y=\mathscr{I}_Y^2/\mathscr{I}_Y^3\otimes \mathbb{Q}_\ell$, $V_Y=\mathscr{I}_Y/\mathscr{I}_Y^3\otimes \mathbb{Q}_\ell$, and $W_Y=\mathscr{I}_Y/\mathscr{I}_Y^2\otimes \mathbb{Q}_\ell$. Note that by Faltings's proof of the Tate conjecture for Abelian varieties \cite[Satz 3]{Faltings83}, it follows that $W_X, W_Y$ are semisimple Galois representations; as $U_X, U_Y$ are quotients of $W_X^{\otimes 2}, W_Y^{\otimes 2}$, respectively, they are also semisimple.

By the observation on semisimplicity in the previous paragraph and Lemma~\ref{lem:unipotent-radical}, the Zariski-closure of the image of Galois in $\GL(V_X)$ is reductive. Hence the Zariski-closure of Galois in $\GL(V_Y)$ is reductive, as a quotient of a reductive group is reductive. Now we conclude by Lemma~\ref{lem:unipotent-radical}.

To prove~\eqref{thm:dt-2}, we proceed analogously, using Lemma~\ref{lem:extn-unipotent-radical} in place of Lemma~\ref{lem:unipotent-radical}. Let $G_X$ be the Zariski-closure of the image of $\pi_1^{\text{\'et}}(X, \bar x)$ in $\GL(V_X)$, and similarly let $G_Y$ be the Zariski-closure of the image of $\pi_1^{\text{\'et}}(Y,  f(\bar x))$ in $\GL(V_Y)$ (note here that we are not taking geometric fundamental groups). Let $S_X$ be the Zariski-closure of the image of $\pi_1^{\text{\'et}}(X_{\overline{k}}, \bar x)$ in $\GL(V_X)$, and let $S_Y$ be the Zariski-closure of the image of $\pi_1^{\text{\'et}}(Y_{\overline{k}},  f(\bar x))$ in $\GL(V_Y)$. 

Unwinding the definition of $J(X), J(Y)$ and applying Lemma~\ref{lem:extn-unipotent-radical}, we see that $J(X)$ (resp.~$J(Y)$) is torsion if and only if $S_X$ (resp.~$S_Y$) is the unipotent radical of $G_X$ (resp.~$G_Y$). By assumption, this is true for $G_X$; now we conclude by the functoriality of $G_X$, $S_X$. That is, $G_Y/S_Y$ is a quotient of $G_X/S_X$, hence reductive.
\end{proof}

\newcommand{\etalchar}[1]{$^{#1}$}
\providecommand{\bysame}{\leavevmode\hbox to3em{\hrulefill}\thinspace}
\providecommand{\MR}{\relax\ifhmode\unskip\space\fi MR }
\providecommand{\MRhref}[2]{%
  \href{http://www.ams.org/mathscinet-getitem?mr=#1}{#2}
}
\providecommand{\href}[2]{#2}

\end{document}